\documentclass{amsart}

\usepackage{amsmath,amssymb,amsthm}
\usepackage{accents}
\usepackage{graphicx}
\usepackage{verbatim}

\newtheorem{theorem}{Theorem}[section]
\newtheorem{proposition}[theorem]{Proposition}
\newtheorem{lemma}[theorem]{Lemma}
\newtheorem{corollary}[theorem]{Corollary}
\newtheorem{definition}[theorem]{Definition}
\newtheorem{question}[theorem]{Question}

\newcommand{\bbz}{{\mathbb Z}}

\newcommand{\crk}{${\mathcal R}_4(K)$}
\newcommand{\rk}{{\mathcal R}_4(K)}
\newcommand{\sK}{{\mathcal S}_4(K)}
\newcommand{\crl}{${\mathcal R}_4(L)$}
\newcommand{\rl}{{\mathcal R}_4(L)}

\newcommand{\rlp}{{\mathcal R}_4(L^\prime)}

\newcommand{\ssk}{\smallskip}
\newcommand{\msk}{\medskip}
\newcommand{\ctln}{\centerline}

\begin{document}

\title{4-moves and the Dabkowski-Sahi invariant for knots}

\author[M. Brittenham]{Mark Brittenham}

\address{Department of Mathematics\\
        University of Nebraska\\
         Lincoln NE 68588-0130, USA}
\email{mbrittenham2@math.unl.edu}

\author[S. Hermiller]{Susan Hermiller}
\address{Department of Mathematics\\
        University of Nebraska\\
         Lincoln NE 68588-0130, USA}
\email{smh@math.unl.edu}

\author[R. Todd]{Robert Todd}
\address{Department of Mathematics\\
        University of Nebraska\\
         Omaha NE 68182-2000, USA}
\email{rtodd@unomaha.edu}

\begin{abstract}
We study the 4-move invariant \crl\ for links in the 3-sphere developed by Dabkowski and Sahi, 
which is defined as a quotient of the fundamental group of the link complement. We develop 
techniques for computing this invariant and show that for several classes of knots it is
equal to the invariant for the unknot; therefore, in these cases the invariant cannot  detect 
a counterexample to the 4-move conjecture.
\end{abstract}

\maketitle
\section{Introduction}\label{sec:intro}

Studying the equivalence classes of knots and links under various types of transformations
on their diagrams
is a well-established subdiscipline of knot theory. This paper concerns
the 4-move, first systematically studied by 
Nakanishi \cite{nakanishi}. The 4-move belongs to the family of $n$-moves,
which involve inserting or deleting $n$ half-twists in series (see
Figure~\ref{fig:4move}), 
and is the only move in the family
whose status as an unknotting operation has not yet been determined.

\ssk

\begin{theorem}\label{thm:basic} \cite{dabprz02},\cite{dabprz11} 
An $n$-move is an unknotting operation, 
i.e., every knot can be transformed to the unknot/unlink by isotopy
and $n$-moves, if $n=1,2$. An $n$-move is not an unknotting operation
if $n=3$ or $n\geq 5$.
\end{theorem}

In 1979 Nakanishi \cite{nakan84},\cite[Problem~1.59~(3)(a)]{kirby} 
conjectured that the 4-move is an unknotting
operation. The conjecture remains open, though it has been verified for several
classes of knots, including 2-bridge knots \cite{przytycki}, 3-braids \cite{przytycki}, 
and all knots with 12 or fewer
crossings \cite{dabjabkhasah}. Because of the theorem above, there is a growing belief 
that the conjecture is false. In fact, a leading candidate for a counterexample has emerged \cite{ask}.

The search for a counterexample to the 4-move conjecture has focused
on constructing  knot and link invariants that are also invariant under 4-moves.
Much of this work has used the fundamental group of the link complement,
or of closely related spaces.

In \cite{dabkowskisahi}, Dabkowski and Sahi define an invariant of a link $L$ in the 3-sphere, \crl ,
which is invariant under 4-moves. 
This invariant \crl\ is a quotient of the fundamental
group of the complement of $L$, $\pi(K)=\pi_1({\mathbb S}^3\setminus L)$, obtained by 
adding relations to a Wirtinger presentation of the link group. 
When $L$ is the unknot then
\crl$\cong \bbz$; thus a counterexample to the 4-move conjecture can be found by finding 
a knot $K$ with \crk $\ncong \bbz$. In what follows we say that a 4-move invariant is 
{\em trivially valued} for a knot $K$, if the invariant for the knot $K$ takes the same 
``value" that the invariant takes on the unknot. That is, we will find a counterexample 
to the 4-move conjecture when we find an invariant and a knot for which the 
invariant is \textit{not} trivially valued.

\begin{figure}
\ctln{\includegraphics[width=3in,height=2in]{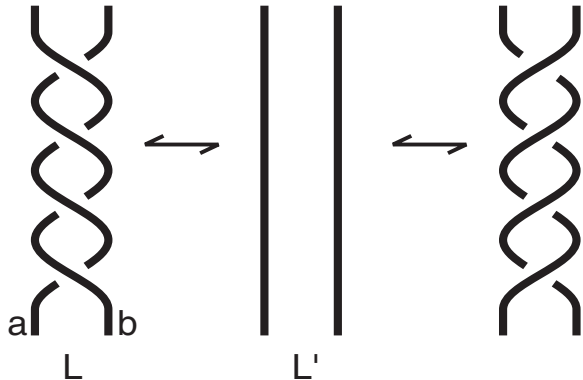}}
\caption{The 4-move}\label{fig:4move}
\end{figure}

In this paper we define a new 4-move invariant for a knot $K$, $\sK$, 
as a quotient of \crk.  
We show that this new invariant is equally as strong a tool when looking for a 
counterexample to the 
4-move conjecture, in the following.

\ssk

\noindent {\bf Corollary \ref{cor:iff}.} {\em  
The following are equivalent: }
\begin{eqnarray*}
 (1) &&\rk\cong\bbz; \mbox{ i.e., } \rk \mbox{\em is trivially valued} \\
 (2) &&\sK \mbox{\em  is finite}\\
 (3) &&\sK \mbox{\em  is abelian}\\
 (4) &&\sK\cong\bbz_2; \mbox{ i.e., } \sK \mbox{\em  is trivially valued.}
 \end{eqnarray*}

\ssk

Analysis of the invariant $\sK$ is more
tractable than the group $\rk$.  For any knot $K$ 
the group $\sK$ 
is a quotient of a Coxeter group, namely a group
generated by finitely many involutions, such that
any pair of the generators generates the dihedral
group of order 8.  Analyzing subgroups generated 
by three of these involutions leads to the following.

\ssk

\noindent {\bf Theorem \ref{thm:bridge3}.}
{\em If $K$ is a knot with bridge number 3, then $\sK$ (and thus \crk)
is trivially valued. More generally,
if $\pi(K)$ is generated by 3 meridians, then $\sK$ and \crk\ are trivially valued.}

\ssk

Another advantage to $\sK$ is that it is more
amenable to algorithmic methods than $\rk$;
computational software 
is often able to determine 
when the invariant $\sK$ is trivially valued, in cases
that the corresponding computations 
applied to \crk\ fail. 
In fact, via computer calculations, we determined that, 
for at least 99.9\% of the 489,107,644  
alternating knots with twenty crossings or less, 
$\sK$, and thus also the Dabkowski-Sahi invariant 
\crk, is trivially valued.

\ssk

\noindent {\bf Theorem \ref{thm:alt18}.}
{\em Among the alternating knots $K$ with up to 20 crossings,
$\sK$ (and thus \crk) is trivially valued, except possibly for 1 knot with 15 crossings, 4 knots with 16 crossings,
41 knots with 17 crossings, and 173 knots with 18 crossings,  31,612 knots with 19 crossings, 
and 274,217 knots with 20 crossings.}

\ssk

The first 5 of these
knots are shown in Figure~\ref{fig:smallest}; 
Gauss codes for those with 17 or fewer crossings can 
be found in an appendix
at the end of the paper. The Gauss codes for the remaining 18, 19, and 20 crossing knots 
can be obtained by contacting the authors. 

\begin{figure}
\ctln{\includegraphics[width=4.5in,height=3in]{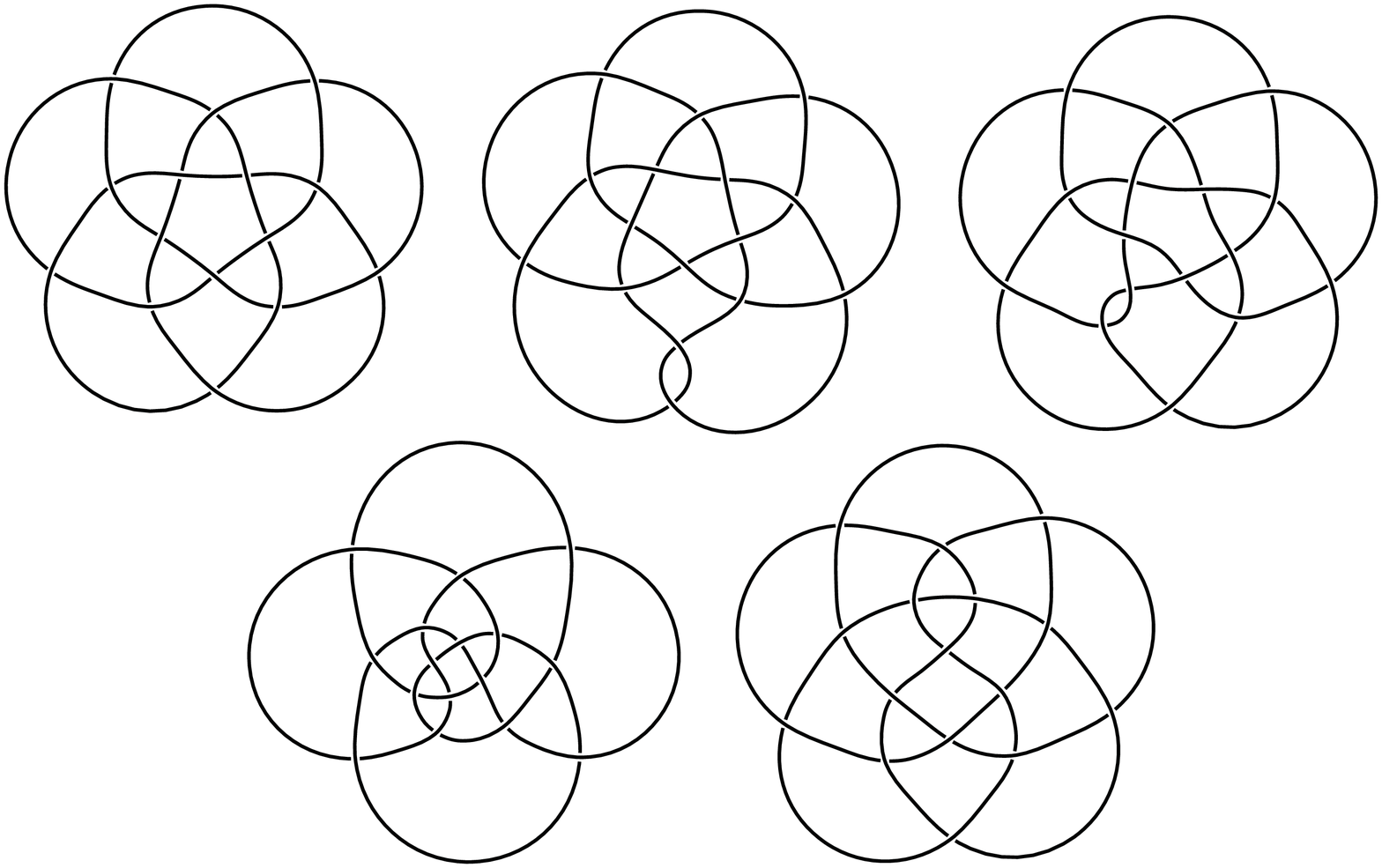}}
\caption{The smallest potential non-trivial examples}\label{fig:smallest}
\end{figure}
\msk

The computations for alternating knots utilized the enumeration by
Flint, Rankin, and Schermann \cite{flintrank1},\cite{flintrank2} 
and their corresponding census as Gauss codes, which are amenable to machine computation. (In addition, the diagrams for 
Figure~\ref{fig:smallest} were generated by their online 
program {\it Knotilus} \cite{knot}.) 
Availability of similar censuses for other classes of knots
would allow application of the same techniques to those classes.

\msk

In outline, in Section \ref{sec:4move} we review the definition of \crl. 
In Section 3 we 
define the group $\sK$ as the top quotient in a normal series for \crk\ 
and show that it is an invariant of knots under 4-moves.
In Section 4 we reduce the problem of showing that 
$\sK$ (and thus \crk) is trivially valued
to showing $\sK$ is \textit{either} finite or abelian. 
In Section 5 we describe several large-scale
computations carried out by the authors, 
and in Section 6 we discuss further avenues of research.

\section{The invariant \crl}\label{sec:4move}

In \cite{dabkowskisahi} Dabkowski and Sahi construct a  a quotient 
of the fundamental group of the exterior of the link $L$, \crl, that is invariant under 4-moves. 
Recall that the Wirtinger presentation for a link group can be obtained from
a diagram of the link: the generators $a_i$ are represented by loops running around the $i$-th unbroken strand
of the diagram, following the righthand rule; this requires an a priori choice of 
orientation to each component of the link. Each crossing provides a relation - either $a_i a_k=a_k a_j$
or $a_k a_i=a_j a_k$, depending upon the handedness of the crossing - where the overstrand labeled
$k$ separates the understrands $i$ and $j$ . As $a_k$ conjugates $a_j$ to $a_i$ (or $a_i$ to $a_j$), 
the generators assigned to each 
component of the link are all conjugate to one another. In the case of knots, as considered here, all of the Wirtinger 
generators are then conjugate to one another.

Starting from a Wirtinger presentation for a link group $\pi(L)$, one may view the invariant \crl\  as follows. Say there is
a  4-move taking $L$ to $L^\prime$ (see Figure~\ref{fig:4move}). Since
$\rl=\rlp$ is a common quotient of both knot groups, a certain word in the 
generators of $\pi(L)$ must be trivial in \crl\ since the corresponding element in $\pi(L^\prime)$ \underbar{is} 
trivial. Depending upon the orientation of the 
strands of the underlying link, the needed relator has the form

\ssk

\ctln{$babab^{-1}a^{-1}b^{-1}a^{-1}$}

\ssk

\noindent where $a$ and either $b$ or $b^{-1}$ are Wirtinger generators, corresponding to the 
two bottom strands of the initial configuration of the simplifying 4-move. 

\msk

To construct their quotient Dabkowski and Sahi~\cite[p.~1266]{dabkowskisahi} 
start with a Wirtinger presentation
$\pi(L)=\pi_1(S^3\setminus L)=\langle X|R\rangle $ and add the relators 

\ssk

\ctln{$babab^{-1}a^{-1}b^{-1}a^{-1}$}

\ssk

\noindent for all $a$ and $b$ that are conjugates 
of an element of $X\cup X^{-1}$;
that is, all $a,b$ in the set
\[{\mathcal C} = \{gxg^{-1} : g\in (X \cup X^{-1})^*, x\in X^{\pm 1}\}\]
(where $(X\cup X^{-1})^*$
denotes the words in the alphabet $X\cup X^{-1}$).
Note that whenever $Y$ is another Wirtinger generating set for $\pi(L)$,
then the subset 
$\{hyh^{-1} : h\in\pi(L), y\in Y^{\pm 1}\}$ of $\pi(L)$
equals ${\mathcal C}$; 
that is, the conjugacy classes of the generating set and their inverses must be the same.
This follows from the fact that any two Wirtinger generators 
associated to the same component of $L$, 
no matter the underlying projection,
are represented by freely homotopic loops in $S^3\setminus L$, 
and therefore are conjugate in $\pi(L)$ 
(using the same orientations of the components).
That is, \crl\ 
does not depend on the initial Wirtinger presentation chosen,
and hence is invariant under Reidemeister moves and so
is an invariant of the link $L$. 
Dabkowski and Sahi then show [DS, Proposition 2.3] that 
\crl\ is, up to isomorphism, unchanged by a 4-move.

Throughout this paper, we consider this invariant in the 
case of a knot $K$.  In this case, all pairs of Wirtinger generators
are conjugate in $\pi(K)$.
A presentation for \crk, with infinitely
many relators, is then given, beginning with a Wirtinger 
presentation $\langle X|R\rangle $ for $\pi(K)$, as 
\[
\rk =\langle X|R\cup R^{\prime\prime}\rangle  = \langle X|R\cup R^{\prime}\rangle ,
\] 
where 
$R^{\prime\prime} := \{(cd)^2(dc)^{-2}\ :\ c,d\in{\mathcal C}\}$ and
\[R^{\prime} := \{(ba)^2(ab)^{-2}\ :\ a\in X, b=ga^\epsilon g^{-1}, 
\epsilon=\pm 1, g\in (X\cup X^{-1})^*\}.\]
These relators are relations $abab=baba$ in the quotient
group \crk.

If K is the unknot then $\rk\cong\bbz$ (choose the projection with no crossings and corresponding Wirtinger
presentation).   Consequently, any knot $K$ that is 4-move equivalent
to the unknot must have $\rk\cong\bbz$. 

Thus, as mentioned above, one may search for a  counterexample to the 4-move conjecture by looking 
for a knot $K$ with $\rk\not\cong\bbz$. 

\section{The invariant $\sK$}\label{sec:tower}

In this section we will focus on knots $K$, and gain a better understanding of 
\crk\ by constructing a 
tower of subgroups of \crk\ and studying the intermediate quotients. This leads
to the definition of the invariant $\sK$.

For the given presentation of \crk, just as for the Wirtinger presentation of $\pi(K)$, 
the sum of the exponents of each relation is 0. Thus, just as for $\pi(K)$,  the 
abelianization of \crk\ is $\bbz$, as all generators are conjugate. It follows 
that if \crk\ is cyclic, then $\rk\cong\bbz$.
%
 However, one can further show the following.

\ssk

\begin{lemma}\label{lem:cyclic}
The group $G=\rk$ is cyclic iff $G/Z(G)$ is cyclic, where $Z(G)$ is the center
of $G$. More generally, $G=\langle X|R\rangle $ is cyclic iff for some subset $S\subseteq Z(G)$
the group $G^\prime=\langle X|R\cup S\rangle $ is cyclic.
\end{lemma}

\begin{proof}
If $G$ is cyclic, then every quotient of $G$ is cyclic, so $G^\prime$ is cyclic. On the other hand,
if $G^\prime$ is cyclic, then $G/Z(G)$ is the quotient of a cyclic group,
so $G/Z(G)=\langle x\rangle $ is also cyclic. Choose any element $\tilde{x}\in G$ which maps to $x$ under the
standard projection $p:G\rightarrow G/Z(G)$. The group $G$ is then generated by $\tilde{x}$ and $Z(G)$; 
given $y\in G$, then $p(y)=x^n$ for some $n$, so that $p(y\tilde{x}^{-n})=1$, and thus
$y\tilde{x}^{-n}=z\in Z(G)$, and $y=z\tilde{x}^n$. But since $\tilde{x}$ commutes with
every element of $Z(G)$ (by definition), $G$ is abelian, so $G$ equals its abelianization,
which as we have already remarked, is $\bbz$. So $G=\rk$ is cyclic.
\end{proof}

In essence, one may add any element of the center of \crk\ 
to its set of relators
without altering the cyclicity of the group. This provides more avenues to determine if
\crk\ itself is $\bbz$. This observation
leads us to look for central elements of \crk.

\begin{proposition}\label{prop:fourth}
If $\pi(K)=\langle X|R\rangle $ is a Wirtinger presentation for the knot group of the 
knot $K$, and $\rk=\langle X|R\cup R^\prime\rangle $ is the corresponding
presentation for \crk, then for every $a,b\in X$ we have $a^4=b^4$ in \crk.
In particular, for every $a\in X$, $a^4\in Z(\rk)$.
\end{proposition}

\begin{proof} 

Since \crk\ is generated
by the Wirtinger generators $b\in X$, to show that
$a^4$ is central it suffices to show that 
$a^4$ commutes with every element $b\in X$.
This will follow from the stronger fact that for
every pair of Wirtinger generators 
$a,b\in X$ we have $a^4=b^4$ in \crk, since $b^4$ clearly commutes 
with $b$, so $a^4$ also commutes with $b$.

Consider the Wirtinger generators $a,b$, and $c$, appearing at a crossing as 
in Figure~\ref{fig:crossing}, with relation $ac=cb$ 
(or $bc=ca$ for the crossing of the opposite sign). 
It suffices to show that  $a^4=b^4$ in \crk\ for this
pair $a,b$, since by induction, traversing the 
knot from undercrossing to undercrossing, one finds
a string of identities $a_1^4=a_2^4=a_3^4=\ldots$ for the successive Wirtinger generators,
showing that all fourth powers of Wirtinger generators are equal.

\begin{figure}
\begin{center}
\includegraphics[width=1.2in,height=1.2in]{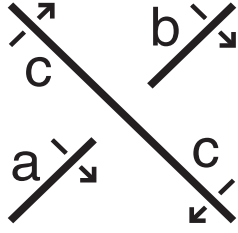}
\caption{Wirtinger generators at a crossing}\label{fig:crossing}
\end{center}
\end{figure}

Since our conclusion is symmetric in $a$ and $b$ it does not
matter which Wirtinger relation actually holds; we will arbitrarily assume that
$ac=cb$. For ease of reading we will follow common notational practice and 
set $A=a^{-1}$, $B=b^{-1}$, and $C=c^{-1}$.   

From $ac=cb$ it follows that $a=cbC$, and so $a^{4}=(cbC)^{4}=cb^{4}C$. Thus $a^{4}=b^{4}$ 
so long as $c$ commutes with $b^{4}$. 
%
Noting that $c$ is also a Wirtinger
generator, and so is conjugate to $b$, consider the following two relations from $R^{\prime}$:

\ssk
\begin{eqnarray*}
(1)&& cbcb=bcbc\\ 
(2)&& c(bcB)c(bcB)=(bcB)c(bcB)c\ .
\end{eqnarray*}
\ssk
\noindent Relation (1) implies that $b$ commutes with $cbc$, and thus $B$ also commutes with $cbc$. Then
\begin{eqnarray*}
(bcbcbcb)cB^{4}& =&  bcbcb[(cbc)(B^{3})]B {\buildrel{(1)}\over{\ =\ }} bcbcb[(B^{3})(cbc)]B = (bcbc)B^{2}cbcB\\
&{\buildrel{(1)}\over{\ =\ }} &(cbcb)B^{2}cbcB = c(bcB)c(bcB) {\buildrel{(2)}\over{\ =\ }} (bcB)c(bcB)c \\
& =&bc(bB^{2})cbcBc = bcb[(B^{2})(cbc)]Bc {\buildrel{(1)}\over{\ =\ }} bcbcbcB^{3}c\\
&=&(bcbcbcb)(B^{4}c)\ .
\end{eqnarray*}

Canceling $(bcbcbcb)$ from both sides we get $cB^4=B^4c$ , so $c$
commutes with $B^4$, and so $c$ commutes with $b^4$, as desired.
\end{proof}

This leads to the tower of subgroups of \crk. For 
the rest of this section, and section 4, fix a Wirtinger 
presentation $\pi(K)=\langle X|R\rangle $ of the knot group.
The bottom of the tower for \crk\ is the subgroup 
\[N := \langle \{a^4\ :\ a\in X\}\rangle \]
Since $a^4\in Z(\rk)$ for all $a\in X$, $N$ is a normal subgroup
of \crk, contained in the center of \crk. 
As $a^4=b^4$ for all $a,b\in X$ by 
Proposition~\ref{prop:fourth}, $N=\langle a^4\rangle $ is cyclic. 
 Lemma~\ref{lem:cyclic} then says that \crk\ is cyclic iff $\rk/N$ is cyclic; note that
if $\rk/N$ is cyclic then $\rk/N\cong \bbz_4$, 
since in this case $\rk=\langle a|\rangle \cong \bbz$.

The top subgroup of the tower for \crk\ is the normal subgroup 
\[H:=\langle \{a^2 : a\in X\}\rangle ^N\] 
(where $\langle~\rangle^N$ denotes the normal subgroup generated
by the set)
of \crk\  
generated by the squares  
of the images of the Wirtinger generators of $\pi(K)$. 
Since $a^4=(a^2)^2$, we have $N\leq H$, giving 
the normal series

\msk
\ctln{$\rk\ \triangleright\ H\ \triangleright N\ \triangleright\ \{1\}$}
\msk

Recall that the set 
${\mathcal C} = \{gag^{-1} : g\in (X \cup X^{-1})^*, a\in X^{\pm 1}\}$ 
is independent of the choice of Wirtinger generating set $X$ for $\pi(K)$.
The groups in this normal series can also be written as
$N = \langle \{c^4 \ : \  c \in {\mathcal C}\}\rangle$
and 
$H = \langle \{c^2 \ : \  c \in {\mathcal C}\}\rangle$;
therefore the subgroups $H$ and $N$ of \crk\ are also independent of the choice of knot projection used
for the Wirtinger presentation. 

\begin{definition}\label{def:sk}
For a knot $K$ with Wirtinger presentation $\pi(K) = \langle X|R\rangle $ we set 
$\sK=\rk/H=\langle X | R\cup R^\prime\cup S\rangle $ where $S=\{ a^2 : a\in X\}$.
\end{definition}

From the discussion above, the following is immediate.

\begin{proposition}
The group $\sK$ is an invariant of the knot $K$, and is 
invariant (up to isomorphism) under 4-moves.
\end{proposition}

Note that if $\rk\cong\bbz$, then under this isomorphism we have
$N=4\bbz$ and $H=2\bbz$. In particular, $\sK=\rk/H\cong\bbz_2$ is a finite, 
abelian 2-group in this case.  The group
$\sK$ is trivially valued if $\sK \cong \bbz_2$.

In general note that in $\sK$ every generator $a \in X$
has order 2, and so we have $a=a^{-1}$.
Moreover, each conjugate of $a$ must also have order 2, and
hence equals its own inverse. 
Therefore in the presentation of $\sK$ in Definition~\ref{def:sk},
the relations $babab^{-1}a^{-1}b^{-1}a^{-1}$ from $R^\prime$
can be replaced by relations
$(ba)^4$ for all
$a\in X$ and $b$ conjugate to $a$.
In particular, $\sK$ is a quotient of
the Coxeter group 
$\langle X \mid \{a^2=1, (ab)^4=1 \mid a,b \in X\}\rangle$.

This enables
us to write a more useful presentation for $\sK$, as:

\begin{lemma}\label{lem:R/Hpres}
$\sK=\rk/H=\langle X | R\cup S\cup T\rangle $,where
$T=\{(awaw^{-1})^4 : a\in X, w\in X^*\}$.
\end{lemma}

The obstruction to \crk\  being cyclic, we shall show
in the next section, lies in the top subgroup $H$ of the tower. 
In particular, it is detected by the quotient group $\sK$. 

\section{Finite $\sK$ is sufficient}\label{sec:finite}

In this section we establish our main result, that for a knot $K$, the quotient $\sK$ is trivially valued 
(that is, is isomorphic to $\bbz_2$)
if and only if \crk\  is trivially valued (i.e., is isomorphic
to $\bbz$). This will be carried out in several steps.
The first step relies on the following theorem of Baer.

\begin{theorem} (\cite{baer}; see \cite{gor}, Ch. 3, Theorem 8.2)
If $G~=~\langle X|U\rangle $ is a finite group whose generators $X$ 
are conjugate in $G$,
and if for every pair $c,d$ in the conjugacy class containing $X$
the subgroup $\langle c,d\rangle \leq G$ generated
by $c$ and $d$ is a 
$p$-group, then $G$ is a $p$-group.
\end{theorem} 

\begin{proposition}\label{prop:2group}
If $\sK$ is finite, then $\sK$ is a 2-group.
\end{proposition}

\begin{proof} 
In $\sK= \langle  X | R\cup S \cup T \rangle $, for any $c,d$ in
the conjugacy class of $X$ we have
the relations $c^2=1$, $d^2=1$, and $(cd)^4=1$ in $\sK$. 
Hence the subgroup $\langle c,d\rangle \leq \sK$ is a 
quotient of the Coxeter group $\langle c,d | c^2,d^2,(cd)^4\rangle $, 
which is the
dihedral group of order 8, and hence is a 2-group. 
So the conditions of Baer's Theorem 
(with $p=2$) are 
met, and $\sK$ is a 2-group.
\end{proof}

\begin{lemma}\label{lem:frat} 
If $G$ is a finite 2-group, whose generators are conjugate to one another,
then $G$ is a cyclic group.
\end{lemma}

\begin{proof} 
This appears to be a standard result (we first learned of it from a discussion on
Math Overload \cite{mathover}); for completeness, the argument is included here. 

The Frattini subgroup $F(G)$
of $G$ is the set of
all of the `non-generators' of $G$, that is, all $c$ such that if $b_1,..,b_n,c$ generate
$G$ then $b_1,..,b_n$  generate $G$.  (See, e.g., \cite{hall}
for basic properties of $F(G)$.) $F(G)$ is a normal subgroup of $G$, and 
$G/F(G)$ is an elementary abelian 2-group. But an elementary
abelian 2-group is a direct sum
of copies of the group $\bbz_2$. Since $G$ is generated by a single conjugacy class, 
the abelianization of $G$ is cyclic, and so the quotient 
$G\rightarrow G/F(G)\cong (\bbz_2)^n$ factors through a cyclic group. 
Thus $n=1$. Choosing an element $t\in G$ that maps to a generator of
$G/F(G)$, $G$ is then generated by $t$ and the finite set $F(G)$.
But then from the definition of $F(G)$, every element of $F(G)$ 
can be inductively removed from the generating set, implying that
$G$ is generated by $\{t\}$, i.e., $G$ is cyclic.
\end{proof}

\begin{corollary}\label{cor:z2}
If $\sK$ is finite, then $\sK\cong\bbz_2$. 
\end{corollary}

\begin{proof}
From Proposition~\ref{prop:2group} and Lemma~\ref{lem:frat},
we have that $\sK$ is a cyclic 2-group; thus we must show that
$\sK$ is not the trivial group.
The abelianization of $\sK$ 
has presentation $\langle X \mid R \cup S \cup T \cup 
\{aba^{-1}b^{-1} \mid a,b \in X\}\rangle$, using the
notation from Lemma~\ref{lem:R/Hpres}.  Now the commutator
relations imply that the Wirtinger relations $R$
can be replaced by relations $a=b$ for all $a,b \in X$,
and the relations $T$ are all redundant.  Hence
the abelianization of $\sK$ has presentation 
$\langle a \mid a^2 \rangle$, and is the group $\bbz_2$.
\end{proof}

Next we turn our analysis to the middle quotient $H/N$ of our normal series.

\begin{lemma}\label{lem:abelian}
$H$ and $H/N$ are abelian.
\end{lemma}

\begin{proof}
Recall that $H$ is the normal closure in \crk\ of the squares 
of the images in \crk\ of the Wirtinger generators $X$.
$H$ is therefore generated by the (possibly infinite) collection of
conjugates of squares of Wirtinger generators of \crk. 
Given a pair of these generators $x,y$ of $H$, we can then set $x=ga^2g^{-1}$ and $y=hb^2h^{-1}$ for 
some $a,b\in X$ and 
$g,h\in\rk$. Now set $p=gag^{-1}$ and $q=hbh^{-1}$; then $x=p^2$ and $y=q^2$. 
Write $P$ for $p^{-1}$ and $Q$ for $q^{-1}$. Note that $p$ and $q$ are conjugate
to elements of $X$.

Using the relations in \crk\
\begin{eqnarray*}
(1)&& QpQp=pQpQ \text{ (so }pq=qPqpQp\text{)}\\ 
(2)&& qPqP=PqPq \text{ (so }qPq=PqPqp\text{)}\\
(3)&& qpqp=pqpq,\\
\end{eqnarray*}
we then find that
\begin{eqnarray*}
xyx^{-1}y^{-1}&=&p(pq)qPPQQ {\buildrel{(1)}\over{\ =\ }}p(qPq)pQ(pq)PPQQ\\
&{\buildrel{(2),(1)}\over{\ =\ }}&p(PqPqp)pQ(qPqpQp)PPQQ {\buildrel{\text{}}\over{\ =\ }}qPqpqpQPQQ\\
&{\buildrel{(3)}\over{\ =\ }}&qPpqpqQPQQ {\buildrel{\text{}}\over{\ =\ }} 1\\
\end{eqnarray*}
Therefore, in $H$, $xyx^{-1}y^{-1} = 1$,  i.e., $yx=xy$. Since this holds for every
pair of generators of $H$, $H$ is abelian.
\end{proof}

We now have the tools to establish the relationship between $\sK$ and $\rk$.

\begin{theorem}\label{thm:finite}
If $K$ is a knot, and if $\sK$ is finite, then $\rk\cong\bbz$.
\end{theorem}

\begin{proof}
Since $\sK=\rk/H$ is finite, $H$ is a finite index subgroup of the
finitely generated group \crk\ and so is finitely generated. Thus 
the generating set $C=\{ga^2g^{-1} : a\in X, g\in\rk\}$ 
for $H$ contains a
finite generating set for $H$.
Notice that in $H/N$, every element of $C$ has order 2, since $(ga^2g^{-1})^2=ga^4g^{-1}=1$, given that $a^4=1$ in $H/N$.
Thus $H/N$ is a finitely-generated
abelian group, whose generators $z_i$ all have order 2. It follows that $H/N$ is isomorphic
to a quotient of $(Z_2)^n$ for some $n$
(the map sending the $i^{th}$ standard generator of $(Z_2)^n$ 
to $z_i$ is a surjective homomorphism).
In particular, $H/N$ is finite with
order a power of 2. But then since $\sK\cong\bbz_2$ by Corollary \ref{cor:z2}, then $\rk/N$ is a finite
group with order a power of 2, and so $\rk/N$ is a finite 2-group. Since it is
generated by the Wirtinger generators, all of which are conjugate, we conclude
from Lemma \ref{lem:frat}
that $\rk/N$ is cyclic. Since $N\subseteq Z(\rk)$, Lemma \ref{lem:cyclic}
implies that $\rk$ is cyclic, 
and so $\rk\cong\bbz$.
\end{proof}

Note that in general if $\sK$ is abelian, then
$\sK$ equals its own abelianization, which was
shown in the proof of Corollary~\ref{cor:z2}
to be the group $\bbz_2$.
Since if $\rk\cong\bbz$ then the quotient $\sK\cong\bbz_2$ 
is both abelian and finite, we have
the following.

\begin{corollary}\label{cor:iff} The following are equivalent:
\begin{eqnarray*}
 (1) &&\rk\cong\bbz; \mbox{ i.e., } \rk \mbox{ is trivially valued.} \\
 (2) &&\sK \mbox{ is finite.}\\
 (3) &&\sK \mbox{ is abelian.}\\
 (4) &&\sK\cong\bbz_2; \mbox{ i.e., } \sK \mbox{ is trivially valued.}
 \end{eqnarray*}
\end{corollary}

Corollary \ref{cor:iff} demonstrates that the non-triviality of \crk\ ``lives'' at the top stage of our
filtration $\rk\triangleright H\triangleright N\triangleright \{1\}$. The main point to this result is that
it appears in practice to be much easier to analyze the presentation
of the group $\sK$ than that of
$\rk$. For example, the fact that in $\sK$ every generator is its own inverse is
in practice a great advantage.

\msk

As an example, a direct computation in GAP shows that the group
\ssk

$G_{3,5}=\langle a_1,a_2,a_3 \mid$

\hfill $ \{a_i^2, 
(a_i(wa_jw^{-1}))^4\ :\ i,j\in \{1,2,3\},
w\in\{a_1,a_2,a_3\}^*,|w|\leq 5\}\rangle,$

\ssk

\noindent (where $|w|$ denotes word length) has order 5192, and in 
particular is finite. But for any knot group generated by at most three meridians,
for example, any 3-bridge knot group, the group $G_{3,5}$ surjects onto $\sK$.
In particular, the map that sends the elements $a_i$
to the three generating meridians is a surjective homomorphism. This implies that
$\sK$ is finite whenever $\pi(K)$ is generated by at most 3 meridians.

\begin{theorem}\label{thm:bridge3}
If $K$ is a knot with bridge number 3, then $\sK$ (and thus \crk)
is trivially valued. More generally,
if $\pi(K)$ is generated by 3 meridians, then $\sK$ and \crk\ are trivially valued.
\end{theorem}

Przytycki has shown that all 2-bridge knots are 4-move equivalent to the unknot \cite{przytycki},
so this computation provides no new information for bridge number less than 3. It is known 
that a knot whose group is generated by two meridians is 2-bridge \cite{bz}; the corresponding
result is not known to be true for three meridians.

\msk

It is tempting to continue this line of reasoning further; we can, for any $n$ and $k$,
define 

\ssk

$G_{n,k}=\langle a_1,\ldots,a_n \mid$

\hfill $\{a_i^2, (a_i(wa_jw^{-1}))^4\ :\ 
i,j\in\{1,\ldots,n\}, w\in\{a_1,\ldots,a_n\}^*,|w|\leq k\rangle $

\ssk

\noindent and define the corresponding group $G_n=G_{n,\infty}$ as a direct limit
of successive quotients, and ask:

\begin{question}\label{question:finite}
Is $G_n$ finite for all $n$?
\end{question}

If the answer to this question is `Yes', then \crk\ is trivially valued (i.e., isomorphic 
to $\bbz$) for all knots $K$. Note that $G_n$ is finite iff $G_{n,k}$ is finite for some $k$, since if $G_n$ is
finite then it has a finite presentation. The generators and relations used in that finite 
presentation can be obtained from our given presentation by finitely many Tietze transformations;
choosing $k$ to be the length of the longest relation used, then $G_{n,k}\cong G_n$ is finite.

This implies that for any $n$ for which the answer to the above question is `Yes' we can in
principle verify this answer by a finite computation. In particular, by running a 
parallel coset enumeration computation
(with a staggered start) on the groups $G_{n,k}$ with ever larger $k$, for any of these groups $G_n$
that are finite, the enumeration is guaranteed to terminate 
(see, for example, \cite[Chapter~5]{heo} for 
details of this procedure).

\section{Large-scale computation of $\sK$}\label{sec:comp}

The original goal of this project was to find a counterexample to the 4-move conjecture,
by discovering a knot $K$ for which $\sK$ was not $\bbz_2$. 
No such example was found. 
(The title of this paper would otherwise have been quite different!) However, 
determining $\sK$ is in practice much more amenable to machine computation than determining
$\rk$, particularly since to show that
$\sK\cong\bbz_2$ it suffices by Corollary \ref{cor:iff}
to show that $\sK$ is either abelian or finite. 
Such computations are much more likely to 
terminate, and are much quicker than computation of $\rk$. 

Strictly speaking, one cannot
hand to a program the infinite set of relators needed to describe $\sK$ or $\rk$;
however, one may truncate the infinite set of relators. 
In analogy with the groups $G_{n,k}$
above, consider the groups $G_k(K)$ defined from a Wirtinger presentation
$\pi(K)=\langle X \mid R\rangle $ for $K$ by 
\[
G_k(K):=\langle X \mid R\cup S \cup R_k\rangle,
\] 
where $S=\{a^2 \mid a \in X\}$ and
 
\ssk

\centerline{$R_k:=\{(awbw^{-1})^4 \mid 
a,b\in X,  w\in X^*, |w|\leq k\}$.}

\ssk

\noindent
Since every relator in this presentation of $G_k(K)$ is
a consequence of the relators
in the presentation of $\sK$ in Definition~\ref{def:sk},
there is a canonical surjective homomorphism
$G_k(K)\twoheadrightarrow\sK$. Therefore, 
in order to show that $\sK$ is finite it suffices to show 
that $G_k(K)$ is finite for some $k$. 
By an identical argument to that given for the groups $G_{n,k}$
above, this is also a necessary condition for the finiteness of $\sK$. 

For example, for the knot $K$
described in \cite{ask} as a likely candidate for a counterexample to the 4-move conjecture,
the computation that $\sK\cong\bbz_2$, and therefore $\rk\cong\bbz$, 
is almost immediate;  its Wirtinger
presentation, together with the relators $a^2$ and a small subset of the relators
$(awbw^{-1})^4$ are sufficient for the program GAP to conclude that the resulting group 
$G_k(K)$, for all sufficiently large $k$,
 has order 2. (In fact, entering the presentation takes far longer than the computation!) Bolstered
by such initial success, we carried out analogous computations on the largest
census of knots at our disposal.

The input data needed for computing $G_{k}(K)$ that is specific to the knot K 
is the Wirtinger presentation, taken from a diagram of the knot.
For alternating knots this presentation can be easily constructed from a Gauss/Dowker
code for the knot \cite{ga}. Recall that the Gauss code of a knot diagram with $n$ crossings
is a string of the integers $\{1,\ldots,n\}$, each occurring exactly twice. The
string is constructed by numbering the crossings $1$ through $n$, and then traveling along the
knot, writing down the crossing numbers encountered in order. For an alternating knot
this is sufficient to construct a Wirtinger presentation for the knot, since the 
additional over/undercrossing information is strictly not needed; we can arbitrarily assume that the
first crossing met is the overcrossing, knowing that succeeding crossings will
alternate. (The `incorrect' choice will lead to the mirror image of the
knot, which has the same knot group.) Since the label for the crossing can be 
imputed to be the labeling
for the overstrand at the crossing, we can determine which understrands meet at a crossing
from the Gauss code. This is illustrated by the example in Figure~\ref{fig:gauss}.

\begin{figure}
\begin{center}
\includegraphics[width=3in,height=2in]{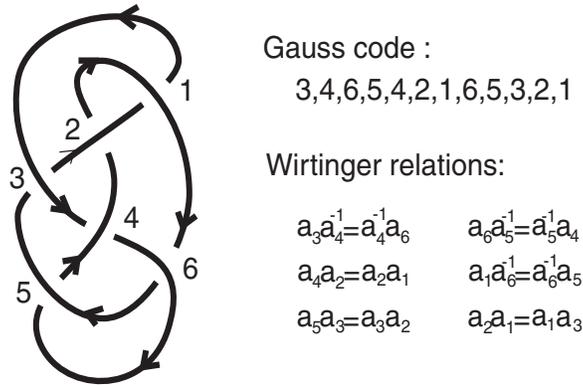}
\caption{From Gauss code to presentation}\label{fig:gauss}
\end{center}
\end{figure}

As can be seen, the trio of numbers centered on an undercrossing (which we may arbitrarily
assume occur at the entries of even index in the Gauss code sequence) reflect the generators,
in order, in a Wirtinger relation. What they do not reflect is the exponent of the 
overcrossing generator, that is, the conjugator, in the relation. But since in 
our quotients $\sK$ and $G_n(K)$
we have $a_i^2=1$ for each generator and hence $a_i=a_i^{-1}$, one can replace $a_i^{-1}$ 
with $a_i$ in these defining
relations without changing the groups that are presented by these ``Wirtinger relations''
together with the relations $\{a_i^2=1 | a_i\in X\}$. These
are the presentations that were extracted from Gauss codes for alternating knots, to use in our
large-scale computations.

Flint and Rankin provide, on their website, the Prime Alternating Knot Generator software \cite{pakg} to generate
the Gauss codes for every alternating knot (without duplication) of whatever number
of crossings is specified by the user. From this, as described above, one can build
a Wirtinger presentation for each knot, and then construct presentations for the 
groups $G_n(K)$ to test for finiteness. 

Two basic algorithms were applied to analyze the group $G_n(K)$. 
First, a coset enumeration algorithm was used to enumerate cosets of the trivial subgroup. 
Second, we used the Knuth-Bendix algorithm 
(see \cite[Chapter~12]{heo} for a description) to look for a confluent rewriting system, 
where the generators were given the order coming from the Gauss code, and words were 
given the ``shortlex" order. In practice, four tests were applied in
succession to winnow the initial list of alternating knots, eliminating those
for which one of these tests determined that $G_n(K)$,
and therefore $\sK$, was finite in turn.  At each step, those knots
for which $G_n(K)$ was
not found to be finite were ``failures'' for that
step, and the next step was applied just to those failure knots.

The general approach to the computations was the following sequence of steps:
\begin{enumerate}
\item Apply a coset enumeration algorithm to enumerate the cosets of the trivial subgroup to 
$G_{0}(K)$. Collect failures.
\item Apply the Knuth-Bendix algorithm to $G_{0}(K)$ for the failures from step 1. Collect failures.
\item Apply the Knuth-Bendix algorithm to $G_{1}(K)$ for the failures in step 2. Collect failures.
\item Apply the Knuth-Bendix algorithm to $G_{2}(K)$ for the failures in step 3. Collect failures.
\end{enumerate}

Several software packages were used to perform these computations. The coset enumeration algorithm was 
applied via GAP \cite{gap4} and its two implementations; the standard implementation via the ``Size" command, 
and the GAP package ACE (Advanced Coset Enumeration). The Knuth-Bendix algorithm was applied via 
the package KBMAG (Knuth-Bendix on Monoids and Automatic Groups) \cite{kbmag} in GAP and via MAF \cite{maf}, KBMAG's 
standalone PC implementation. In each case either a memory limit or a time limit was used to delineate 
success from failure. That is, either the computation finished or was abandoned after some limit was 
reached.  In each case in which GAP (or one of its C++ packages, ACE or KBMAG)  was used the standard 
memory limits (e.g. table size for coset enumeration) were used to delineate success from failure. 
In the cases where MAF was used, either a time limit was set (e.g. 5 minutes when considering 
$G_{1}(K)$), or an ad-hoc approach to determine that no progress was being made (this was done in 
step 4) was used to determine failure. 
  
For alternating knots with 18 crossing or less all computations were performed on a personal 
computer, except for step 1 for 18 crossing knots, which were completed on \textit{Firefly}, 
a 5600 core AMD cluster managed by the Holland Computing Center at the University of Nebraska at Omaha. 
In all of these cases, GAP's 
standard implementation of the coset enumeration algorithm was applied 
to the presentations of $G_{0}(K)$ 
via the ``Size" command. Steps 2,3, and 4 were performed on a personal computer using MAF. 
After step 1, a relatively small list 
of these knots were left ($6681=8+82+1572+5019$ in total with 15 through 18 crossings). As noted above, 
in steps 2 and 3 a time limit of 5 minutes was set. 
After step 3 the list consisted of $763=1+27+201+534$ knots with 15 through 18 
crossings. Thus, when applying step 4 to this small list, no time limit was set. In each of these 
cases, either the computation was successful or was halted by hand when the computation seemed to 
be making no progress (signaled by a long period of adding new equations without any reduction).

For the 19 and 20 crossing knots, steps 1 and 2 were performed on Tusker, a 40 TF cluster 
consisting of 106 Dell R815 nodes using AMD 6272 2.1GHz processors, also managed by the 
Holland Computing Center, at the University of Nebraska-Lincoln. 
Step 1 used the C++ coset enumeration implementation ACE, that comes with the GAP 
installation (though GAP was not initiated as an interface) and step 2 used the C++ implementation 
of KBMAG that comes with the GAP installation (again, GAP was not used as an interface). 
After step 1 of the 80,689,811 19-crossing knots, all but approximately 450,000 were shown 
to have $G_0(K)\cong\bbz_2$ and of the 397,782,507 20-crossing alternating knots, 
all but approximately 4,500,000 were shown to have $G_0(K)\cong\bbz_2$. Step 2 reduced 
these numbers to 31,612 and 274,217, respectively. We have yet to apply steps 3 and 4 to 
the current list of failures for 19 and 20 crossing knots. 

\begin{theorem}\label{thm:alt18}
Among the alternating knots $K$ with up to 20 crossings,
$\sK$ (and thus \crk) is trivially valued, except possibly for 1 knot with 15 crossings, 4 knots with 16 crossings,
41 knots with 17 crossings, and 173 knots with 18 crossings,  31,612 knots with 19 crossings, 
and 274,217 knots with 20 crossings.
\end{theorem}

Thus among the alternating knots with 20 or fewer crossings, all but at most 0.06\% have
trivially-valued Dabkowski-Sahi invariant. For those with 18 or fewer
crossings, this percentage is 0.0021\%. We anticipate
that completion of steps 3 and 4 would significantly 
reduce the first percentage.

\section{Concluding thoughts and future directions}\label{sec:future}

Since every knot is 4-move equivalent to an alternating knot (choose crossings
whose change would result in an alternating knot, and replace
them with three crossings of the opposite sign), if there is a counterexample to the
4-move conjecture there is an alternating knot counterexample. One can view the above 
computations as
either providing evidence in support of the 4-move conjecture or providing a larger set 
of possible counter examples, depending on one's own opinion of the truth or falsity 
of the 4-move conjecture. In fact, attempts by the authors to 4-move reduce some of 
the remaining knots identified by Theorem \ref{thm:alt18} to the unknot have never succeeded.

However, these computations may suggest that the invariants \crk\ and $\sK$ cannot 
detect a counter-example to Nakanishi's 4-move conjecture. Thus we pose the following question:

\ssk

\begin{question}\label{qes:cyclic} Are \crk\ and $\sK$ trivially valued for every knot $K$?
\end{question}

\ssk

The authors have formulated several lines of attack for this question; none as yet 
can be carried to their conclusion. For example, if one could show that in $\sK$ every 
element has order dividing four then Question \ref{qes:cyclic} could be answered in 
the affirmative, as $\sK$ would be the quotient of
a finitely generated Burnside group of exponent 4, all of which are known \cite{burn4}
to be finite, and hence (by Corollary~\ref{cor:iff})
$\sK$ would be isomorphic to $\bbz_2$.

Ultimately, Question \ref{qes:cyclic} is really a question of group theory. Given any 
group $G$ that is normally generated by a single element $x$, then  by 
Johnson's Theorem \cite{Jo} this group is a quotient of a knot group via a surjection 
that sends a Wirtinger generator for the knot group (for some diagram) to $x$. 
However, if $G = \langle X \mid R \rangle$ 
is such a group, then one may consider the corresponding tower of 
subgroups constructed in Section \ref{sec:tower} outside of the context of a knot diagram,
namely 

\ssk

\noindent $R(G):=\langle X|R \cup R^{\prime\prime} \rangle 
\triangleright H := \langle \{(gag^{-1})^2 : a\in  X, g\in G\}\rangle^N$

\hfill $\triangleright N := \langle \{(gag^{-1})^4 : a\in  X, g\in G\}\rangle^N\triangleright 1$,

\ssk

\noindent where $R^{\prime\prime} = \{(cd)^2(dc)^{-2} \mid
c,d \in {\mathcal C}\}$.  That is, $R(G)$ is the 
quotient of $G$ analogous to the quotient \crk\ of $\pi(K)$, 
defined by modding out by the same set of relators. 
Since \crk, for any knot $K$ built via Johnson's Theorem, would 
then surject onto $R(G)$, 
finding such a group G for which $R(G)$ is not cyclic 
would then imply the existence of a counterexample to the 4-move conjecture.
As for knots, the cyclicity of $R(G)$ can be determined by the finiteness
of $S(G):=R(G)/H$.

Thus one sees that Question \ref{qes:cyclic} is  equivalent to the following question:
\ssk
\begin{question}\label{qes:cyclic2} Suppose that $G=\langle X|R\rangle $ with 
$|X| < \infty$ such that the set $X$ lies in a single conjugacy class in $G$. If
$S=\{x^2 : x\in X\}$ and $T = \{\{(xgyg^{-1})^4 : x,y\in X, g\in X^*\}$,
is it always true that the quotient group $\langle X|R\cup S \cup T\rangle $ 
of $G$ is cyclic? Equivalently,
must this quotient be finite? Equivalently, must it be abelian?
\end{question}


\vfill
\eject

\appendix 

\section{Perl code for the algorithms in Section \ref{sec:comp}}

I: Perl code computing the size of $G_0(K)$, calling GAP and ACE:

\medskip

\begin{verbatim}

#!/usr/bin/perl
use strict; use warnings;
open ACE, "| /[absolute path to]/ace" or die print "can't open ace\n";
$"=",";  my @squares=(); my $xings=20;  # Don't forget to change this!!!
for my $i (1..$xings) { push(@squares,"$i^2"); }  my @prodpower4=();
for my $k1 (1..$xings) { for my $k2 (1..$xings)
{ push(@prodpower4,"(($k1)($k2))^4"); } }
for(my $i=197;$i<=200;$i++){  my $zero_num = sprintf("%04d", $i);
my $file="20xing".$zero_num;   # this is the filename for the codes
my $Stats="/work/unknots/rtodduno/20X/Stats.1/20xing.stats".$zero_num;
my $fail="/work/unknots/rtodduno/20X/Fail.1/20xing.fail".$zero_num;
my $cx="/work/unknots/rtodduno/20X/PCX.1/20xing.pcx".$zero_num;
open(ST,">>",$Stats);  print ACE "ao: $Stats;";
open(GC,"<$file") or die print "can't open $file\n";
my $s=1;  while(my $line=<GC>) { chomp $line;
my @knot=split(",",$line);  my $twoxings=2*$xings;  my @wirtinger=();
for my $j (1..$xings-1) {  my $overarc=$knot[2*$j-1];
my $dsarc=$knot[(2*$j)-2]; my $usarc=$knot[(2*$j)];
my $overgen="$overarc";    my $dsgen="$dsarc";
my $usgen="$usarc";   my $rel="($dsgen)($overgen)($usgen)($overgen)";
push(@wirtinger,$rel);  }
my $lastxing=$knot[$twoxings-1];  my $firstxing=$knot[0];
my $second2lastxing=$knot[$twoxings-2];
my $toprel="($second2lastxing)($lastxing)($firstxing)($lastxing)";
push(@wirtinger,$toprel);  my @allrelations=();
push(@allrelations,@squares);   push(@allrelations,@prodpower4);
push(@allrelations,@wirtinger);
my $command="Group Generators:$xings;Group Relators:
                                 @allrelations;Subgroup: trivial;start;";
print ACE "$command \n";  $s++   }
close ST;  close GC;  open(ST,"<$Stats");  open(FAIL,">>",$fail);
open(PCX,">>",$cx);  my $t=1;
while(my $line=<ST>) { chomp($line); my @info=split(/ /,$line);
if($info[0] eq "OVERFLOW") { print FAIL "$t\n"; }
elsif($info[0] eq "INDEX") { if($info[2] != 2)
{ print PCX "$t\n"; } } $t++ }  close ST;  close FAIL;  close PCX;  }  
close ACE;

\end{verbatim}

\vfill
\eject

II: Perl code computing an FCRS for $G_0(K)$, using KBPROG:

\medskip

\begin{verbatim}

#!/usr/bin/perl
use strict;  use warnings;my $xings=20;	# Don't forget to change this!!!
my @gens=();  for my $t (1..$xings)  { push(@gens,"f".$t.",F".$t) }
my @ingens=();  for my $q (1..$xings) { push(@ingens,"F".$q.",f".$q) }
$"=","; my @squares=();  for my $i (1..$xings)
{ push(@squares,"[f"."$i^2,IdWord]"); } my @prodpower4=();
for my $k1 (1..$xings) { for my $k2 (1..$xings)
{ push(@prodpower4,"[(f".$k1."*f".$k2.")^4,IdWord]"); } }
my $stillbad="stillbad".$ARGV[1];
open(SB,">>",$stillbad) or die print "can't open $stillbad\n";
open(GC,"<$ARGV[0]") or die print "can't open $ARGV[0]\n";
my $s=1; while(my $line=<GC>)
{ my $gapfile="20xingmaf".$ARGV[1]."/maf".$xings."_".$s.".txt";
  open(GAP,">>",$gapfile) or die print "can't open $gapfile \n";
  chomp $line; my @knot=split(",",$line); my $twoxings=2*$xings;
  my @wirtinger=(); for my $j (1..$xings-1)
{ my $overarc=$knot[2*$j-1]; my $dsarc=$knot[(2*$j)-2];
  my $usarc=$knot[(2*$j)]; my $overgen="f".$overarc; my $dsgen="f".$dsarc;
  my $usgen="f".$usarc;
  my $rel="[".$dsgen."*".$overgen."*".$usgen."*".$overgen.",IdWord]";
  push(@wirtinger,$rel); }
my $lastxing=$knot[$twoxings-1]; my $firstxing=$knot[0];
my $second2lastxing=$knot[$twoxings-2];
my $toprel="[f".$second2lastxing."*f".$lastxing."*f".$firstxing.
                                              "*f".$lastxing.",IdWord]";
push(@wirtinger,$toprel); my @allrelations=(); push(@allrelations,@squares);
push(@allrelations,@prodpower4); push(@allrelations,@wirtinger);
print GAP "_RWS:=rec(isRWS :=true,ordering";
print GAP " :=\"shortlex\",generatorOrder:=[@gens],";
print GAP "inverses:=[@ingens],equations:=[@allrelations]);";
system("/[absolute path to]/kbprog -silent $gapfile");   $s++;
my $outfile=$gapfile.".reduce"; my $extrafile1=$gapfile.".kbprog";
my $extrafile2=$gapfile.".kbprog.ec";
open(OUT,"<$outfile") or die print "can't open $outfile\n";  my @rec=<OUT>;
chomp $rec[12];   #  print "$rec[12]\n";
my $numst=substr($rec[12],-1);
if($numst !=2){ print SB "$line \n"; }
system("rm  $gapfile $outfile $extrafile1 $extrafile2");	
}

\end{verbatim}

\vfill
\eject

\section{Gauss codes for potential counterexamples to the 4-move conjecture}

The alternating knots $K$ with 17 or fewer crossings, 
for which the procedures 
of Section \ref{sec:comp} have {\it not} shown that $\rk\cong\bbz$:

\medskip

\noindent 1,2,3,4,5,6,7,8,9,3,10,11,6,12,13,9,2,14,11,5,15,13,8,1,14,10,4,15,12,7\par
\noindent 1,2,3,4,5,6,7,8,9,10,11,12,6,1,13,9,14,15,12,5,2,13,8,16,15,11,4,3,10,14,16,7\par
\noindent 1,2,3,4,5,6,7,8,9,10,11,5,12,1,8,13,14,11,4,3,15,9,13,16,6,12,2,15,10,14,16,7\par
\noindent 1,2,3,4,5,6,7,8,9,10,11,12,6,13,2,9,14,15,12,5,16,3,10,14,8,1,13,16,4,11,15,7\par
\noindent 1,2,3,4,5,6,7,8,9,10,11,5,12,1,8,13,14,11,4,15,2,9,13,16,6,12,15,3,10,14,16,7\par
\noindent 1,2,3,4,5,6,7,8,9,10,11,12,13,14,6,15,4,11,16,17,8,1,15,5,12,16,10,3,2,9,17,13,14,7\par
\noindent 1,2,3,4,5,6,7,8,9,10,11,12,13,14,8,3,2,9,15,16,12,5,6,13,17,15,10,1,4,7,14,17,16,11\par
\noindent 1,2,3,4,5,6,7,8,9,10,11,12,13,14,8,1,15,5,12,11,4,3,16,9,14,17,6,15,2,16,10,13,17,7\par
\noindent 1,2,3,4,5,6,7,8,9,10,11,12,13,14,15,3,10,9,4,16,14,1,17,11,8,5,16,15,2,17,12,7,6,13\par
\noindent 1,2,3,4,5,6,7,8,9,10,11,5,12,1,8,13,14,11,4,3,15,9,13,16,17,14,10,15,2,12,6,17,16,7\par
\noindent 1,2,3,4,5,6,7,8,9,10,11,12,6,1,13,14,2,5,12,15,16,9,14,13,8,17,15,11,4,3,10,16,17,7\par
\noindent 1,2,3,4,5,6,7,8,9,10,11,12,13,14,15,9,16,1,6,13,12,5,2,16,8,17,14,11,4,3,10,15,17,7\par
\noindent 1,2,3,4,5,6,7,8,9,10,11,12,6,5,13,14,2,15,8,16,12,13,17,3,15,9,10,1,14,17,4,7,16,11\par
\noindent 1,2,3,4,5,6,7,8,9,10,11,12,6,13,14,3,2,15,8,16,12,5,17,14,15,9,10,1,4,17,13,7,16,11\par
\noindent 1,2,3,4,5,6,7,8,9,10,11,9,12,13,6,14,15,3,10,11,2,16,14,5,17,12,8,1,16,15,4,17,13,7\par
\noindent 1,2,3,4,5,6,7,8,9,10,8,11,12,5,13,14,2,15,11,16,6,13,17,3,15,9,10,1,14,17,4,12,16,7\par
\noindent 1,2,3,4,5,6,7,8,9,10,11,12,13,14,15,9,2,5,12,16,17,15,8,1,6,13,16,11,4,3,10,17,14,7\par
\noindent 1,2,3,4,5,6,7,8,9,10,11,12,6,13,2,1,14,7,15,11,4,16,13,14,8,17,10,3,16,5,12,15,17,9\par
\noindent 1,2,3,4,5,6,7,8,9,10,11,5,12,13,2,9,14,15,10,3,13,16,6,17,15,14,8,1,16,12,4,11,17,7\par
\noindent 1,2,3,4,5,6,7,8,9,10,11,5,12,13,8,14,15,11,4,3,16,1,13,7,14,17,10,16,2,12,6,15,17,9\par
\noindent 1,2,3,4,5,6,7,8,9,10,11,12,13,14,8,1,15,5,12,16,17,9,2,15,6,13,16,11,4,3,10,17,14,7\par
\noindent 1,2,3,4,5,6,7,8,9,10,11,5,4,12,13,1,8,14,15,11,12,3,16,7,14,17,10,13,2,16,6,15,17,9\par
\noindent 1,2,3,4,5,6,7,8,9,10,11,5,12,13,8,14,15,11,4,3,16,1,13,7,17,15,10,16,2,12,6,17,14,9\par
\noindent 1,2,3,4,5,6,7,8,9,10,11,12,13,14,8,15,2,5,12,16,17,9,15,1,6,13,16,11,4,3,10,17,14,7\par
\noindent 1,2,3,4,5,6,7,8,9,10,11,12,6,5,13,14,2,9,15,16,12,13,17,3,8,15,10,1,14,17,4,7,16,11\par
\noindent 1,2,3,4,5,6,7,8,9,10,4,3,11,12,8,13,14,15,10,11,16,1,6,14,17,9,12,16,2,5,15,17,13,7\par
\noindent 1,2,3,4,5,6,7,8,9,10,11,12,6,13,14,3,2,9,15,16,12,5,17,14,8,15,10,1,4,17,13,7,16,11\par
\noindent 1,2,3,4,5,6,7,8,9,10,11,5,4,12,13,1,8,14,15,11,12,16,2,7,14,17,10,13,16,3,6,15,17,9\par
\noindent 1,2,3,4,5,6,7,8,9,10,11,12,13,14,15,9,2,3,10,16,12,5,17,1,8,15,16,11,4,17,6,13,14,7\par
\noindent 1,2,3,4,5,6,7,8,9,10,11,12,13,5,14,15,16,1,10,13,4,17,15,7,8,16,2,11,12,3,17,14,6,9\par
\noindent 1,2,3,4,5,6,7,8,9,10,11,3,12,5,13,9,14,1,15,12,4,11,16,14,8,17,6,15,2,16,10,13,17,7\par
\noindent 1,2,3,4,5,6,7,8,9,10,11,12,2,13,6,14,8,15,12,3,16,5,10,17,15,1,13,16,4,11,17,9,14,7\par
\noindent 1,2,3,4,5,6,7,8,9,10,11,5,12,13,2,14,10,15,6,12,4,16,14,1,17,7,15,11,16,3,13,17,8,9\par
\noindent 1,2,3,4,5,6,7,8,9,10,11,5,12,3,13,14,10,7,15,12,4,16,14,1,17,15,6,11,16,13,2,17,8,9\par
\noindent 1,2,3,4,5,6,7,8,9,10,11,12,4,13,2,14,10,7,15,5,12,16,14,1,17,15,6,11,16,3,13,17,8,9\par
\noindent 1,2,3,4,5,6,7,8,9,10,11,12,8,13,6,14,15,3,10,11,2,16,14,5,17,9,12,1,16,15,4,17,13,7\par
\noindent 1,2,3,4,5,6,7,8,9,10,11,12,4,13,14,1,10,7,15,5,12,16,2,14,17,15,6,11,16,3,13,17,8,9\par
\noindent 1,2,3,4,5,6,7,8,9,10,11,3,12,7,13,14,10,15,16,1,8,13,6,17,4,11,15,16,2,12,17,5,14,9\par
\noindent 1,2,3,4,5,6,7,8,9,10,11,3,12,5,13,14,8,1,15,12,4,11,16,9,14,17,6,15,2,16,10,13,17,7\par
\noindent 1,2,3,4,5,6,7,8,9,10,11,3,12,5,13,9,14,1,15,12,4,16,10,14,8,17,6,15,2,11,16,13,17,7\par
\noindent 1,2,3,4,5,6,7,8,9,10,11,3,12,7,13,14,15,11,2,16,8,13,6,17,4,15,10,1,16,12,17,5,14,9\par
\noindent 1,2,3,4,5,6,7,8,9,10,11,12,4,13,6,14,8,15,2,11,16,5,13,17,15,1,10,16,12,3,17,7,14,9\par
\noindent 1,2,3,4,5,6,7,8,9,10,11,12,4,13,14,7,8,15,2,11,16,5,17,14,15,1,10,16,12,3,13,17,6,9\par
\noindent 1,2,3,4,5,6,7,8,9,10,11,3,12,5,13,14,8,1,15,11,4,12,16,7,14,17,10,15,2,16,6,13,17,9\par
\noindent 1,2,3,4,5,6,7,8,9,10,11,5,12,13,2,14,10,7,15,12,4,16,14,1,17,15,6,11,16,3,13,17,8,9\par
\noindent 1,2,3,4,5,6,7,8,9,10,11,3,12,5,13,7,14,1,10,15,4,12,16,14,8,17,15,11,2,16,6,13,17,9\par

\end{document}